\documentclass[11pt,a4paper]{amsart}
\usepackage{graphicx}
\usepackage{amssymb}
\input xy
\xyoption{all}
\usepackage[all]{xy}
\usepackage{hyperref}

\newtheorem{thm}{Theorem}[section]

\newtheorem{lem}[thm]{Lemma}

\theoremstyle{definition}

\newtheorem{rem}[thm]{Remark}

\numberwithin{equation}{section}
\newcommand{\QQ}{\mathbb Q}

\newcommand{\ZZ}{\mathbb Z}
\newcommand{\CC}{\mathbb C}
\newcommand{\PP}{\mathbb P}

\newcommand{\lra}{\longrightarrow}

\newcommand{\ra}{\rightarrow}

\newcommand{\cL}{\mathcal{L}}

\newcommand{\cF}{\mathcal{F}}
\newcommand{\cH}{\mathcal{H}}

\newcommand{\tv}{\tilde{\varphi}}

\newcommand{\cG}{\mathcal{G}}
\newcommand{\cM}{\mathcal{M}}

\newcommand{\cO}{\mathcal{O}_C}
\newcommand{\cP}{\mathcal{P}}

\DeclareMathOperator{\End}{{End}}

\DeclareMathOperator{\Pic}{Pic}

 \DeclareMathOperator{\im}{Im}
 \DeclareMathOperator{\Div}{{Div}}

\begin{document}

\title[ ]{ The Brill-Noether curve and Prym-Tyurin varieties}%
\author{ Angela  Ortega}
              
\address{Angela Ortega \\ Institut f\" ur Mathematik, Humboldt Universit\"at zu Berlin \\ Germany}
\email{ortega@math.hu-berlin.de}

%\thanks{Research partially supported by the Sonderforschungsbereich 647  ``Raum - Zeit - Materie''}
%\subjclass{14H40, 14H30}

%\date{\today }
%\dedicatory{ }%
%\commby{ }%
% ----------------------------------------------------------------
\begin{abstract} 

We prove that the Jacobian of a general curve $C$ of genus $g=2a+1$, with $a\geq 2$, can be realized as 
a Prym-Tyurin variety for the Brill-Noether curve $W^1_{a+2}(C)$. As consequence of this result we are able to
compute the class of the sum of secant divisors of the curve $C$, embedded with a complete linear series
$g^{a-1}_{3a-2}$.

\end{abstract}

\maketitle

\section{Introduction}

Consider a smooth general curve $C$ of genus $g\geq 5$. The locus $W^r_d(C)$
parametrizing line bundles $L$ of degree $d$ over $C$ with $h^0(L)\geq r+1$, is an irreducible variety of
dimension equal to the Brill-Noether number $\rho=\rho(g,r,d)$. In particular,  $ W^r_d(C)$  is a smooth 
curve when $\rho=1$. In the case of $g=5$, $r=1$, the involution $\iota: L \mapsto \omega_C\otimes L^{-1}$, induces an automorphism of the curve $W^1_4(C)$, which is of genus 11. Since $C$ is general, the quotient map
$W^1_4(C) \ra W^1_4(C)/ \iota$ is an \'etale double covering over a curve of genus 6. If $P$ denotes the Prym 
variety associated to this covering, it is known that $P$ is isomorphic to the Jacobian $JC $ as a principally polarized
abelian variety (\cite{ma}). The main result of this paper shows that this situation generalizes to curves of higher odd genus, obtaining in this way Prym-Tyurin varieties. 
\vspace{.2cm}

Recall that a principally polarized abelian variety (ppav)  $(P,\Xi)$ is a Prym-Tyurin 
variety if there exists a smooth projective curve 
$X$, such that $P$ is an abelian subvariety of the Jacobian $JX$ and the restriction of the principal polarization of $JX$
to $P$ is algebraically equivalent to  $e \Xi$, where $e \in \ZZ_{>0}$ is the exponent of $P$ in $JX$. In that case, we say that $P$ is a Prym-Tyurin variety for the curve $X$ with exponent $e$. 
\vspace{.2cm}

Let $g=2a+1$, for $a\geq 2$. The locus $W^1_{a+2}(C)$ is a smooth curve, which from now on
will be called the {\it Brill-Noether curve}.  We define a correspondence $\gamma$ on $W:=W^1_{a+2}(C)$, hence an 
endomorphism of the Jacobian  $JW$, by means of the multiplication of sections. More precisely,
$$
L \mapsto \gamma(L):=\big\{ L' \in W \mid H^0( L) \otimes H^0(\omega_C\otimes (L')^{-1}) \ra H^0(\omega_C \otimes L\otimes (L')^{-1}) \textrm{ is not injective} \big\}.
$$
Let $P:=\im (1- \gamma) \subset JW^1_{a+2}(C)$. 
We prove the following theorem.
\begin{thm} \label{main}
Let $C$ be a general curve of genus $g=2a+1$. 
The subvariety $P:= \im (1-\gamma) $ is a Prym-Tyurin variety for the Brill-Noether curve $W^1_{a+2}(C)$
 of  exponent the Catalan number
 $$ 
\frac{(2a)!}{a! (a+1)!}. 
$$
Moreover,  $P \simeq JC$ as principally polarized abelian varieties.
\end{thm}

This result can also be interpreted from the point of view of enumerative geometry.
It is reasonable to expect that, under suitable generality assumptions, a linear series $L \in W^r_d(C)$ has finitely many 
$(2r-2)$-secant $(r-2)$-planes that is, divisors $D \in C^{(2r-2)}$ such that $h^0(L(-D)) \geq 2$. 
In that case the number of secants is computed by the Castelnuovo formula (\cite[Chapter VIII]{acgh}).
Then one can associate to every linear series $g^r_d$ an element of $\Pic(C)$, namely the class of the sum of the secant divisors. 
By the results of Ciliberto (\cite{c}) it is natural to expect that this class should depend only on the canonical divisor and the $g^r_d$. 
For instance, when $\rho(g,r,d)=0$ one can assign to the curve the class of the sum of the elements in 
$W^r_d(C)$. In this situation, Franchetta's conjecture implies that the sum is a multiple of the canonical bundle (see \cite{b}). 
\vspace{.2cm}

For $a\geq 4$, the residual linear system of  $L \in W^{1}_{a+2}(C)$ defines an embedding $C \hookrightarrow \PP^{a-1}$,
whose image admits finitely many $(2a-4)$-secant  $(a-3)$-planes. These secants are in bijection with the elements
of $\gamma(L)$ by setting $L'= \omega_C\otimes L^{-1} (-D) \in \gamma(L) $, where $D $ is the divisor defined by a secant plane of the embedded curve. As an application of the Theorem \ref{main}, we are able to determine the class in 
$\Pic (C)$ of the sum of the secant divisors. 

\begin{thm} 
 Let $C$ be a general smooth curve of genus $2a+1$. For any line bundle $L \in W^1_{a+2}(C)$ we have that
 $$
\bigotimes_{L' \in \gamma(L)}  L' =  \omega_C^{\alpha} \otimes L^{1-e},
 $$ 
 where $e=\frac{(2a)!}{a! (a+1)!}$ and $\alpha= \frac{a+2}{4a}( g(W)-1 - e(g-1))$.
 \end{thm}

A principally polarized abelian variety can always be realized as a Prym-Tyurin variety for some curve, 
but with a very large exponent  (see \cite[Corollary 12.2.4]{bl}). In fact, the curve  for which a ppav is a Prym-Tyurin variety is not uniquely determined. It is an open problem to find, for a fixed $g$,  the smallest integer $m$ such that any ppav of dimension 
$g$ is a Prym-Tyurin variety of exponent $e \leq m$.
For instance, Mumford's results (\cite{m}) show that the general ppav of dimension $g=4, 5$
is a Prym-Tyurin variety of exponent 2 for a curve of genus $2g-1$. Another example of 
Prym-Tyurin varieties of small exponent can be found in \cite{lr}, where the authors exhibit
a family of Prym-Tyurin varieties of dimension 6 and exponent 6.

\section{The Brill-Noether curve}

Let $C$ be a general curve of genus $g = 2a+1$ satisfying Petri's theorem. Let $\omega_C$ denote the canonical line
bundle on $C$. Consider the Brill-Noether locus $W:=W^1_{a+2}(C)$ consisting of line bundles $L$ on $C$ 
of degree $a+2$, with $h^0(C,L)\geq 2$.  Since $C$ is general and $\rho(g, 1, a+2) =1$, the locus $W$ is a smooth irreducible curve naturally embedded in $\Pic^{a+2}(C)$.  The genus of the Brill-Noether curve is computed by the formula (\cite[Theorem 4]{eh}):
\begin{equation} \label{genus}
g(W) = \frac{a}{a+2}\cdot \frac{2g!}{ a! (a+1)!} +1.
\end{equation}
We fix a point  $ L_0 \in \Pic^{a+2}(C)$ and consider the embedding 
$$
\varphi : W \ra JC,  \quad L\mapsto L\otimes L_0^{-1}.
$$ 
\begin{lem}
The curve $W$ generates $JC$ as an abelian group.
\end{lem}
\begin{proof}
The embedding $\varphi: W \hookrightarrow  JC$ induces a morphism $\tv: JW \ra JC $.
It suffices to show that $\tv$ is surjective. It has been shown in \cite{fl}, that the induced map
$$
\varphi_*: H_1(W, \ZZ) \lra H_1(JC, \ZZ) \simeq H_1(C, \ZZ)
$$
is surjective. This map corresponds to the rational representation of $\tv$ and it determines it completely.
Hence, $\tv$ is surjective.
\end{proof}
Thus we have a short exact sequence
\begin{equation} \label{ses}
0 \lra K^1_{a+2}(C) \lra JW^1_{a+2}(C) \stackrel{\tv}{\lra} JC \lra 0,
\end{equation}
where $\tv$ is the map which takes a class of equivalence of divisors of degree zero in $W^1_{a+2}$ to its linear
equivalence class as a divisor on the curve $C$. The following result is proved in 
\cite[Theorem 1.1]{cht}.
\begin{thm} \label{cht}
 For a general curve $C$ af genus $g\geq 3$, the abelian variety $K^1_{a+2}(C)$ is connected and has no non-trivial endomorphisms which are rationally determined. 
\end{thm}
By rationally determined we mean defined over the field of rational functions of $\cM_{g,1}$, the 
moduli space of smooth pointed curves of genus $g$.

Let us denote 
$\theta_C: JC \stackrel{\simeq}{\ra} \widehat{JC}$ (respectively $\theta_W$) the principal polarization of $JC$
(respectively that of $JW$).
By dualizing the exact sequence (\ref{ses}), we find that 
$$
\varphi^* = \theta_W^{-1} \circ \hat{\tv} \circ \theta_C : JC \ra JW,
$$
is an embedding since $\hat{\tv}$  is also one (see \cite[Prop. 2.4.2]{bl}). We shall show that the image of $\varphi^*$ defines an abelian subvariety 
of $JW$, which is a Prym-Tyurin variety for $W$. A polarized abelian variety $(P, \Xi)$ is a Prym-Tyurin 
variety for a curve $C$ if  there is an embedding $i_P: P \hookrightarrow JC$ such that $i_P^* \Theta \equiv e \Xi$; the integer $e$ is called the exponent of $P$.  
We will use Welters' criterion for Prym-Tyurin varieties (\cite[Theorem 12.2.2]{bl}).

\begin{thm} \label{Welters}(Welters' Criterion). 
Let $(P, \Xi)$ be a ppav of dimension $g$ and $C$ and smooth curve. Then $(P, \Xi)$ is a 
Prym-Tyurin variety of exponent $e$ for $C$ if and only if it exists a morphism $\phi: C \ra P$ such that \\
a) $\phi^*: P \ra JC$ is an embedding, \\
b) $\phi_*[C] = \frac{e}{(g-1)!} \bigwedge^{g-1}[\Xi]$ in $H^{2g-2}(P, \ZZ)$.
\end{thm}

\begin{thm}
The Jacobian $JC$ is a Prym-Tyurin variety for $W$ of exponent the Catalan number
\begin{equation}\label{exp}
e = \frac{(2a)!}{a! (a+1)!}.
\end{equation}
\end{thm}
\begin{proof}

We apply the Criterion \ref{Welters} to the embedding $\varphi^*: JC \ra JW$.
It suffices to show that $\varphi_*[W]$ has the required cohomology class. The class of the curve $W$ in 
$H^{2g-2}(\Pic^{a+2}(C), \ZZ)$ is given by (\cite[p. 320]{acgh})
\begin{equation} \label{classBN}
 [W^1_{a+2}(C)] =\frac{1}{a!(a+1)!} \bigwedge^{g-1} [\Theta_C].
\end{equation}
Hence
$$
\varphi_*[W] =\frac{1}{a!(a+1)!} \bigwedge^{g-1} [\Theta_C] =  \frac{e}{(2a)!} \bigwedge^{g-1} [\Theta_C]
$$
in $H^{2g-2}(JC, \ZZ)$.
\end{proof}

\section{A correspondence on the Brill-Noether curve }

We define the  following correspondence on the Brill-Noether curve $W$:
$$
 \gamma : L \mapsto \{ L' \in W \mid \mu: H^0(C, L) \otimes H^0( C, \omega_C\otimes (L')^{-1}) \ra H^0(C, \omega_C \otimes L\otimes 
 (L')^{-1}) \textrm{ is not injective} \},
$$ 
where  $\mu$ denotes the multiplication of sections.
It has been shown in \cite{fo} that this correspondence is non-empty for any $a \geq 2$.  
The correspondence $\gamma$ defines an endomorphism  (denoted by the same symbol) $\gamma \in \End(JW)$ by 
$$
\left[ \sum n_i L_i \right] \mapsto \left[ \sum n_i \gamma({L_i}) \right],
$$
where $L_i$ are points on the curve $W$ (corresponding to line bundles of degree $a+2$). 
Using the base-point-free-pencil trick, one checks that $L' \in \gamma(L)$ if and only if   
$H^0(C, \omega_C\otimes L^{-1} \otimes (L')^{-1}) \neq 0$. So,
we can rewrite the correspondence $\gamma$ as
$$
\gamma(L)= \{ L' \in W \mid H^0(C, \omega_C\otimes L^{-1} \otimes (L')^{-1}) \neq 0\}.
$$
From this description follows that  $\gamma$ is symmetric. Moreover, since $C$ is general the Gieseker-Petri Theorem 
(\cite[p. 215]{acgh}) ensures that the multiplication map  
$$
H^0(C,L) \otimes H^0(C, \omega_C \otimes L^{-1}) \ra H^0(C,  \omega_C)
$$
is injective for any line bundle $L \in W^1_{a+2}(C)$. Thus the correspondence $\gamma$ has no fixed points, i.e. 
$\gamma$ does not intersect the diagonal $\Delta \subset W \times W$. 
This also shows that the induced endomorphism of $JW$, is not a multiple of the identity, since these endomorphisms are induced by divisors of the form $n\Delta$, for $ n \in \ZZ$. 

For instance, for $a=2$ the 
correspondence $\gamma$ induces an involution on the curve $W$ of genus 11, namely $ \iota: L \mapsto \omega_C\otimes L^{-1}$.  It is known that the corresponding Prym variety associated to the \'etale double covering $W \ra W / \iota$ is an abelian subvariety of $JW$ of dimension 5 isomorphic to the Jacobian of $C$ (\cite{ma}).
 
\begin{lem} \label{castel}
Let $a \geq 3$ and $g=2a+1$.  The degree of the correspondence $\gamma$ is given by the Castelnuovo number
\begin{equation} \label{deg}
C(a-1, 3a-2, 2a+1) = \sum_{i=0}^{a-2} \frac{(-1)^i}{a+2} {a \choose a-2-i} {2a-i \choose a-1-i}.
\end{equation}
\end{lem}

\begin{proof}
Let $L \in W^1_{a+2}(C)$ and set $M=\omega_C \otimes L^{-1} \in W^{a-1}_{3a-2}(C)$. An element $L' \in W^1_{a+2}(C)$ is in 
$\gamma(L)$ if and only if $H^0(M\otimes (L')^{-1}) \neq 0$, that is, if $M\otimes (L')^{-1} = \cO (D) $ for an effective divisor $D$ of degree $2a-4$. 
So $L'$ is of the form 
$M(-D)$ with $h^0(M(-D))\geq 2$. Hence the degree is given by the degree of the degeneracy locus in $C_d$ 
(the $d$-symmetric power of $C$) 
of the divisors of degree $d=2a-4$ that impose at most $d-r = a-2$ conditions on $|M|$.
Thus one can interpret the degree of $\gamma$ as the number of the $(2a-4)$-secant 
$(a-3)$-planes in the linear system $|M|$. For the general curve there are finitely many of such
$(a-3)$-planes (\cite{f}) and their number is given by the Castelnuovo formula 
(\cite[Chapter VIII]{acgh}).
\end{proof}
 
 The endomorphism $\gamma$ also defines a map $W^1_{a+2}(C) \ra \Pic^m(C)$, where $m:= (a+2)(\deg\gamma)$
 by considering $\gamma(L)$ as a tensor product of line bundles on $C$. More precisely, if $D_i \in \Div^{2a-4}(C)$, for $i=1, \ldots, \deg \gamma$, are the secant divisors of the image of $C$ in 
$|\omega_C\otimes L^{-1}|^* \simeq \PP^{a-1}$, we set $L_i:=M(-D_i) \in  \Pic^{a+2}(C)$ and $\gamma(L)$ can be viewed as 
the line bundle
\begin{equation}\label{tensor}
 \bigotimes_{i=1}^{\deg \gamma} L_i  \in \Pic^m(C).
 \end{equation}

We denote $\cP$  the Zariski open subset consisting of all equivalence classes  $(C,x) \in  \cM_{g,1}$, 
with $C$ a curve having no non-trivial automorphisms and satisfying the Petri 
condition. There exists a smooth scheme $\cG^r_d$ and  a morphism
$\cG^r_d \ra \cP $ such that the fiber over any closed point $t \in \cP$ is isomorphic to $G^r_d(C)$, the variety 
parametrizing all the $g^r_d$'s on $C$. For $a\geq 4 $, set $r=1$, $d=a+2$ and $\cG:= \cG^1_{a+2}$. 
Now, let $\cH$ be the Hilbert scheme of curves of degree $3a-2$ and of genus $2a+1$ in $\PP^{a-1}$ 
and $\cH_1$ the open set of 
a component of $\cH$ with general moduli, parametrizing curves without nontrivial automorphisms.   
For every point $z=((C,x),D) \in \cG$ denote by $\Gamma \subset \PP^{a-1}$ the image of $C$ by the residual series    
$|\omega_C \otimes L^{-1}|$, with $L= \cO(D)$ and by $[\Gamma]$ the corresponding point in $\cH_1$. 
Let $\cH_z$ be a closed subset of $\cH_1$ given by the orbit of $[\Gamma]$ under the action of $PGL(a, \CC)$ (see \cite[\S 3]{c}). The map $\gamma: z \mapsto \gamma(L) $ induces a regular section
 $\cH_z \ra \Pic^m(C)$. By varying the curve $C$ we obtain a rationally determined line bundle $\cL$ on the universal family 
 $\cF$ over $\cH_1$,  such that the restriction of $\cL$ to the fiber over $z$ is isomorphic to $\gamma(L) \in \Pic^m(C)$, 
 where $\gamma(L)$ is the tensor product \eqref{tensor}. 
As a consequence of the Theorem \ref{cht} one has the following result 
(\cite[Theorem 1.2]{cht}).

\begin{thm}  \label{linebundles}
Let $\cH $ be any component of the Hurwitz scheme of coverings of $\PP^1$ of degree $d$ and genus $g\geq 3$
containing curves with general moduli and with $\rho=1$. Then the group of rationally determined line bundles of the 
universal family over $\cH$ is generated by the relative canonical bundle and the hyperplane bundle.
\end{thm}
 
It follows that there exist integers $\alpha,\beta$ such that 
\begin{equation}\label{le}
\gamma (L) = \omega_C^{\otimes \alpha} \otimes L^{\otimes \beta} \in \Pic^m(C). 
\end{equation}
We are able to deduce the coefficients $\alpha$ and $\beta$ as an application of Theorem \ref{prym}.
Set $P:=\im (1-\gamma) \subset JW$. On the light of the Theorem \ref{cht}, one does not expect other
subvarieties of $JW$, other that the obvious ones. More precisely, we prove: 

\begin{thm} \label{prym}
The subvariety $P$ is isomorphic to $JC$. In particular, $P$ is a Prym-Tyurin variety for $W$ of exponent $e$.
\end{thm}
\begin{proof}
Consider the map $\tv_{|P}: P \ra JC$ and suppose it is non-zero. Then by Theorem \ref{cht}, $\tv_{|P}$ is an isogeny.
The embedding $\varphi^*: JC \ra JW$ gives then an isomorphism $JC \simeq P $. In particular, $(P, \Theta_{C|_{P}})$
is a Prym-Tyurin variety of exponent $e$ for W.  If the restriction of $\tv$ to $P$ is zero, the complementary 
subvariety $Z$ of $P$ with respect to $\Theta_W$ is isogenous to $JC$, via the restriction $\tv_{|Z}: Z \ra JC $.
In this case $\varphi^*(JC) = Z$ and $Z$ is a Prym-Tyurin of exponent $e$ for $W$. Moreover,
$Z = \im (e-1 + \gamma)$. Using the formula in  \cite[Corollary 5.3.10]{bl}, one computes that 
$$
\dim Z = \frac{(e-1)g(W)  +  \deg \gamma}{e}.
$$
Since $\dim Z = \dim JC = g$, we have $(e-1)g(W)  + \deg \gamma =eg $. By Lemma \ref{rel} we obtain that
$$
(e-2)g(W) = - 2 \deg \gamma,
$$
which is a contradiction since $e\geq 2$ and $\deg \gamma > 0$.
Therefore $JC \simeq \im (1- \gamma)$.
\end{proof}

\begin{lem} \label{rel}
The equation $g(W) -\deg (\gamma) = e g$ holds.
\end{lem}
\begin{proof}
A direct computation.
\end{proof}

\section{The equivalence class of the sum of secants to a curve}

For any line bundle $L \in W^1_{a+2}(C)$, consider the product
 $$
\gamma(L) = \bigotimes_{i=1}^{\deg \gamma} L_i
$$
as defined in \S 3.
\begin{thm} \label{sum}
 Let $C$ be a general smooth curve of genus $2a+1$. For any line bundle $L \in W^1_{a+2}(C)$ we have that
 $$
 \gamma(L)= \bigotimes_{i=1}^{\deg \gamma} L_i = \omega_C^{\alpha} \otimes L^{1-e},
 $$ 
 where $e=\frac{(2a)!}{a! (a+1)!}$ and $\alpha= \frac{a+2}{4a}( g(W)-1 - e(g-1))$.
 \end{thm}

 \begin{proof}
The norm-endomorphism corresponding to the subvariety $P \subset JW$ is $1-\gamma$. 
It satisfies  $(1-\gamma)^2 =e (1-\gamma)$, or equivalently,  the quadratic equation

\begin{equation} \label{eq}
(  1 - e - \gamma)(1-\gamma) = \gamma^2 +(e-2) \gamma -(e-1)=0
\end{equation}
on the Jacobian $JW$. Consider the projection  $\tv$ from $P= \im(1-\gamma)$ to $JC$. Fix
$M \in W^1_{a+2}(C)$.  Then, by (\ref{le}), there exists an integer $\beta$ such that
$$
\tv (1- \gamma)(L-M) = L \otimes M^{-1} \otimes L^{-\beta} \otimes M^{\beta} .
$$ 
Since the relation  (\ref{eq}) holds on $JC$ as well, we obtain
\begin{eqnarray*}
(  1-e -\gamma)(  L \otimes M^{-1} )^{\otimes 1-\beta} &= &(L \otimes M^{-1})^{\otimes (1-e)(1-\beta)} 
\otimes (L \otimes M^{-1})^{\otimes -(1-\beta)\beta}
\\
 &=& (L \otimes M^{-1})^{\otimes (1-\beta)(1-e - \beta) }\\
 &=& \cO  
\end{eqnarray*}
 for all $L \in W$. Therefore  $(1-\beta)(1-e + \beta)=0$. If $\beta=1$,  $\tv(1-\gamma) =0 $, which is a contradiction 
 to the fact that $\tv_{|P}$ is surjective. Hence $\beta = 1-e $.  In order to compute the value of $\alpha$ one
compares the degrees in the equation (\ref{le}) and uses Lemma \ref{rel}.
\end{proof}

For example, for a general line bundle $L\in W^1_{5}(C)$ on a curve $C$ of genus $7$, the
image of the map $\phi : C\ra |\omega_C \otimes L^{-1}|^* $  is a plane curve with 8 nodes. Let $p_i, q_i$, for $i=1, \ldots, 8$, denote the pre-images of the nodes. Set $M:= \omega_C \otimes L^{-1} \in W^2_7(C)$. Hence 
$$
\gamma(L)= \bigotimes_{i=1}^{8} M (-p_i-q_i)  \ \in \Pic^{40}(C).
$$
By the adjunction formula we have that
$$
\omega_C = M^4( - \sum_{i=1}^{8} p_i+q_i ),
$$
that is, 
$$
\bigotimes_{i=1}^8 M (-p_i-q_i) =  \omega_C^{ 5} \otimes L^{-4},
$$
which is predicted by Theorem \ref{sum} since the Catalan number is equal to 5.  A less trivial example is the case of 
a general curve $C$ of genus 9 embedded in $\PP^3$ by the linear system $|M|$, with $M \in W^3_{10}(C)$. The space curve
admits 43  4-secant lines, the genus of the curve $W^3_{10}(C)$ is $169$
and the exponent of the Prym-Tyurin variety is 14. Let us denote $D_i \in \Div^4(C)$ the corresponding divisors. By Theorem \ref{sum} we obtain $\alpha=21$, $\beta=-13$ and
$$
\bigotimes_{i=1}^{43} \cO(D_i) =  M^{30} \otimes \omega_C^{-8}.
$$

\begin{rem}
For a general curve $C$, the subring of $H^*(C_2, \QQ)$ generated by the fundamental classes
of algebraic cycles on $C_2$ is generated by the class of a fiber of the projection $\pi_1:C_2 \ra C$ 
and the diagonal (\cite[p. 359]{acgh}). In the situation of the Brill-Noether curve, such subring  of $H^*(W_2, \QQ)$ has an 
extra generator induced by the correspondence $\gamma \subset W \times W$. 
\end{rem}

\begin{rem} 
It would be interesting to study the properties of the curve $W,$ for instance determine its gonality or 
if it has a special Brill-Noether behavior.
\end{rem}

{\bf{Acknowledgements.}} I would like to thank to A. Beauville, C. Ciliberto, G. Farkas, E. Izadi, G.-P. Pirola, O. Serman and A. Verra for stimulating conversations. This research is partially supported by the Sonderforschungsbereich 647  ``Raum - Zeit - Materie''.

\end{document}